\documentclass[11pt]{article}
\usepackage[margin=1in]{geometry}
\usepackage{amsmath,amssymb,amsthm,mathtools}
\usepackage{bm}
\usepackage{enumitem}
\usepackage{float}
\usepackage[colorlinks=true,linkcolor=blue,citecolor=blue,urlcolor=blue]{hyperref}

\newtheorem{theorem}{Theorem}[section]

\newtheorem{lemma}[theorem]{Lemma}

\newtheorem{assumption}[theorem]{Assumption}

\theoremstyle{definition}

\numberwithin{equation}{section}

\newcommand{\E}{\mathbb E}
\newcommand{\Pp}{\mathbb P}
\newcommand{\R}{\mathbb R}
\newcommand{\op}{\operatorname{op}}
\newcommand{\argmin}{\operatorname*{arg\,min}}

\title{A GPH-Filtered Hannan--Rissanen Information Criterion for ARFIMA Order Selection}
\author{Chunhao Cai\\
School of Mathematics (Zhuhai), Sun Yat-sen University\\
\texttt{caichh9@mail.sysu.edu.cn}}
\date{\today}

\begin{document}
\maketitle

\begin{abstract}
We propose a simple two-stage order selector for finite-order ARFIMA models.  First, a preliminary log-periodogram estimate of the memory parameter is used to fractionally filter the data.  Second, a Hannan--Rissanen residual construction is applied to the filtered series, and the autoregressive and moving-average orders are selected by a generalized information criterion over a growing candidate rectangle.  The search bounds are allowed to satisfy \(P_n,Q_n\to\infty\), whereas the true orders remain fixed and finite.  The penalty is allowed to be larger than the ordinary BIC penalty so that it dominates the error introduced by preliminary long-memory estimation and by the Hannan--Rissanen residual approximation.  We prove a uniform residual-variance approximation over the growing rectangle and combine it with a population separation argument between the true finite ARMA representation and underfitted alternatives.  The resulting generalized information-criterion selector is consistent.
\end{abstract}



\section{Introduction}

Huang, Chan, Chen and Ing \cite{HuangChanChenIng2022} prove consistent BIC-type order selection for ARFIMA processes by minimizing a conditional-sum-of-squares criterion over a prescribed finite rectangle
\[
    0\le p\le P,
    \qquad
    0\le q\le Q.
\]
This formulation is natural after finite upper bounds \(P\) and \(Q\) have been fixed.  It is less satisfactory as a statistical-computational principle when a data analyst is not willing to regard any fixed pair of upper bounds as final before seeing larger samples.  Although the true AR and MA orders are finite, the candidate region itself may reasonably be allowed to expand slowly with the sample size.  The present paper studies this growing-sieve version of the order-selection problem,
\[
    0\le p\le P_n,
    \qquad
    0\le q\le Q_n,
    \qquad
    P_n,Q_n\to\infty.
\]

Our procedure is deliberately simple.  We first estimate the memory parameter by a log-periodogram estimator of Geweke--Porter-Hudak type \cite{GewekePorterHudak1983,Robinson1995LogPeriodogram,Velasco2000}.  We then fractionally filter the data, construct Hannan--Rissanen residuals \cite{HannanRissanen1982}, and minimize a generalized information criterion over the growing rectangle.  The resulting algorithm avoids a full nonlinear search over \((p,q,d)\).  Its proof is also transparent: the preliminary filtering error and the Hannan--Rissanen residual error are controlled first, and these estimates are then combined with a uniform residual-variance approximation and an information-criterion comparison.

We use the standard ARFIMA notation; see, for example, Granger and Joyeux \cite{GrangerJoyeux1980}, Hosking \cite{Hosking1981}, Beran \cite{Beran1994}, and Sowell \cite{Sowell1992}.  Throughout the paper, the observed process \((Y_t)_{t\in\mathbb Z}\) satisfies the finite-order ARFIMA equation
\begin{equation}
    \Phi_0(B)(1-B)^{d_0}Y_t=\Theta_0(B)\varepsilon_t,
\end{equation}
where \(d_0\in(-1/2,1/2)\), the unknown orders are \((p_0,q_0)\), and
\begin{equation}
    \Phi_0(z)=1-\sum_{j=1}^{p_0}\alpha_{0,j}z^j,
    \qquad
    \Theta_0(z)=1+\sum_{k=1}^{q_0}\beta_{0,k}z^k.
\end{equation}
The zero-order conventions are \(\Phi_0\equiv1\) if \(p_0=0\) and \(\Theta_0\equiv1\) if \(q_0=0\).  After the true fractional filter is applied,
\begin{equation}
    X_t=(1-B)^{d_0}Y_t,
\end{equation}
the short-memory component \((X_t)\) satisfies
\begin{equation}
    \Phi_0(B)X_t=\Theta_0(B)\varepsilon_t.
\end{equation}

The standing assumptions are as follows.
\begin{assumption}\label{ass:truth}
The following conditions hold.
\begin{enumerate}[label=(\roman*)]
    \item \(d_0\in[d_-,d_+]\subset(-1/2,1/2)\).
    \item \(\Phi_0\) and \(\Theta_0\) have no zeros on \(|z|\le 1+\eta_0\) for some \(\eta_0>0\).
    \item \(\Phi_0\) and \(\Theta_0\) are coprime.
    \item If \(p_0>0\), then \(\alpha_{0,p_0}\ne0\).  If \(q_0>0\), then \(\beta_{0,q_0}\ne0\).
    \item \((\varepsilon_t)\) are independent, mean-zero, sub-Gaussian random variables with \(\E\varepsilon_t^2=\sigma_\varepsilon^2\in(0,\infty)\).
\end{enumerate}
\end{assumption}

Assumption \ref{ass:truth} implies that \((X_t)\) is causal and invertible, in the standard ARMA sense; see Brockwell and Davis \cite[Chapter 3]{BrockwellDavis1991}.  In particular, it has the causal representation
\begin{equation}\label{eq:causal-representation}
    X_t=\sum_{\ell=0}^{\infty}\psi_\ell\varepsilon_{t-\ell},
    \qquad
    \sum_{\ell=0}^{\infty}|\psi_\ell|r^\ell<\infty
\end{equation}
for some \(r>1\), and the infinite autoregressive representation
\begin{equation}\label{eq:infinite-ar-representation}
    X_t=\sum_{\ell=1}^{\infty}\varphi_{0,\ell}X_{t-\ell}+\varepsilon_t,
    \qquad
    \sum_{\ell=1}^{\infty}|\varphi_{0,\ell}|R^\ell<\infty
\end{equation}
for some \(R>1\).  This exponential summability is the only tail property of the inverse AR representation used below.

Let \(P_n\) and \(Q_n\) denote the candidate-order upper bounds, and write
\begin{equation}
    M_n=P_n\vee Q_n,
    \qquad
    s_n=P_n+Q_n.
\end{equation}
We assume that \(P_n\ge p_0\) and \(Q_n\ge q_0\) eventually.  Under the growth, separation and penalty conditions stated later, the proposed selector satisfies
\[
    \Pp\{(\widehat p_n,\widehat q_n)=(p_0,q_0)\}\longrightarrow1.
\]
Section \ref{sec:method-main} defines the GPH-filtered Hannan--Rissanen information criterion and states the main theorem.  Section \ref{sec:simulation} reports a simulation study following the same computational order as the theoretical construction.  Section \ref{sec:approximation} proves the approximation lemmas, and Section \ref{sec:proof-main} proves the main theorem.

\section{The GPH--HR--GIC selector and main result}\label{sec:method-main}

\subsection{Preliminary GPH filtering}

We first estimate the memory parameter.  Let \(m_n\) be the number of low Fourier frequencies used in the log-periodogram regression, and assume throughout this subsection that
\begin{equation}\label{eq:gph-bandwidth}
    m_n\to\infty,
    \qquad
    m_n/n\to0 .
\end{equation}
Put
\begin{equation}
    \lambda_j=2\pi j/n,
    \qquad
    x_j=-\log\{4\sin^2(\lambda_j/2)\},
    \qquad
    1\le j\le m_n .
\end{equation}
For the periodogram \(I_Y(\lambda_j)\) of \((Y_t)\), define the Geweke--Porter-Hudak slope estimator
\begin{equation}\label{eq:gph-estimator}
    \widetilde d=
    \frac{\sum_{j=1}^{m_n}(x_j-\bar x_m)\log I_Y(\lambda_j)}
    {\sum_{j=1}^{m_n}(x_j-\bar x_m)^2},
    \qquad
    \bar x_m=m_n^{-1}\sum_{j=1}^{m_n}x_j .
\end{equation}
The proof below only uses the following preliminary rate.  The condition \eqref{eq:gph-bandwidth} is the low-frequency bandwidth condition for the log-periodogram regression; the finite-order causal-invertible ARMA component in Assumption \ref{ass:truth} supplies the required second-order smoothness of the short-memory spectral factor at the origin.
\begin{lemma}\label{lem:gph-rate}
Suppose Assumption \ref{ass:truth} holds and \(m_n\) satisfies \eqref{eq:gph-bandwidth}.  Then the estimator \eqref{eq:gph-estimator} satisfies
\begin{equation}\label{eq:gph-rate}
    |\widetilde d-d_0|=O_p(a_{d,n}),
    \qquad
    a_{d,n}=m_n^{-1/2}+m_n^2n^{-2}.
\end{equation}
In particular, \(a_{d,n}\downarrow0\).
\end{lemma}
\begin{proof}
This is the standard log-periodogram regression rate: in the Gaussian case it follows from Robinson \cite[Sections 2--3]{Robinson1995LogPeriodogram} and the bias-variance calculation of Hurvich, Deo and Brodsky \cite[Section 2]{HurvichDeoBrodsky1998}; for linear processes that are not necessarily Gaussian, the corresponding log-periodogram consistency theory is given by Velasco \cite[Section 2]{Velasco2000}.  Under Assumption \ref{ass:truth}, the spectral density admits
\[
    f_Y(\lambda)=|1-e^{-i\lambda}|^{-2d_0}f_X(\lambda),
    \qquad
    \log f_X(\lambda)=\log f_X(0)+O(\lambda^2),
    \qquad \lambda\to0,
\]
so the second-order spectral-bias term is \(O(m_n^2n^{-2})\), while the low-frequency periodogram term is \(O_p(m_n^{-1/2})\).  This gives \eqref{eq:gph-rate}.
\end{proof}

With this preliminary estimate, define the feasible fractionally filtered series by
\begin{equation}
    \widetilde X_t=(1-B)^{\widetilde d}Y_t.
\end{equation}

\subsection{Filtered ARMA-type selector}

After the preliminary estimate \(\widetilde d\) has been obtained, we apply an ARMA-type information criterion to the filtered series \(\widetilde X_t\).  The construction below is purely computational.  The ideal short-memory objects used to justify the procedure are introduced later, in the proof sections.

The moving-average part of an ARMA fit requires innovation proxies.  We obtain them by a preliminary high-order autoregression on the filtered series.  Let \(\mathcal I_n\) be a common trimmed sample used in all least-squares criteria,
\begin{equation}
    \mathcal I_n=\{t_N,t_N+1,\ldots,n\},
    \qquad
    N=|\mathcal I_n|\asymp n .
\end{equation}
The trimming only removes the initial observations needed for fractional filtering and lags.  Residuals used with lags in the candidate regressions are computed on the corresponding enlarged lag range; this changes only endpoint observations and preserves \(N\asymp n\).

Choose a deterministic preliminary autoregressive order \(h_n\to\infty\).  For \(t\in\mathcal I_n\), set
\begin{equation}
    \widetilde X_{t,h_n}
    =(\widetilde X_{t-1},\ldots,\widetilde X_{t-h_n})^\top .
\end{equation}
Define the high-order least-squares coefficient vector by
\begin{equation}
    \widehat\varphi
    =(\widehat\varphi_1,\ldots,\widehat\varphi_{h_n})^\top
    \in
    \argmin_{\varphi\in\R^{h_n}}
    N^{-1}\sum_{t\in\mathcal I_n}
    (\widetilde X_t-\widetilde X_{t,h_n}^\top\varphi)^2 .
\end{equation}
The corresponding Hannan--Rissanen residuals are
\begin{equation}\label{eq:hr-residual}
    \widetilde e_t
    =\widetilde X_t-\widetilde X_{t,h_n}^\top\widehat\varphi,
    \qquad t\in\mathcal I_n .
\end{equation}
They are used as feasible innovation proxies in the ARMA-type regressions below.

For each candidate order \((p,q)\), define the feasible regressor vector
\begin{equation}
    \widetilde Z_t(p,q)=
    (\widetilde X_{t-1},\ldots,\widetilde X_{t-p},
    -\widetilde e_{t-1},\ldots,-\widetilde e_{t-q})^\top,
\end{equation}
with the empty-block convention when \(p=0\) or \(q=0\).  If
\(\gamma=(a_1,\ldots,a_p,c_1,\ldots,c_q)^\top\), then
\begin{equation}
    \widetilde X_t-\widetilde Z_t(p,q)^\top\gamma
    =\widetilde X_t-
    \sum_{j=1}^{p}a_j\widetilde X_{t-j}
    +\sum_{k=1}^{q}c_k\widetilde e_{t-k}.
\end{equation}
Thus the negative signs in the last block of \(\widetilde Z_t(p,q)\) only encode the ARMA sign convention.

Fix a large compact radius \(K<\infty\), and put
\begin{equation}\label{eq:K-ball}
    \mathcal K_{p,q}(K)=\{\gamma\in\R^{p+q}:\|\gamma\|_2\le K\}.
\end{equation}
For candidate order \((p,q)\), define the Hannan--Rissanen residual variance by
\begin{equation}
    \widehat\sigma_n^2(p,q)=
    \inf_{\gamma\in\mathcal K_{p,q}(K)}
    N^{-1}\sum_{t\in\mathcal I_n}
    (\widetilde X_t-\widetilde Z_t(p,q)^\top\gamma)^2.
\end{equation}
The compact ball stabilizes singular or nearly singular overfitted regressions.  In applications \(K\) may be taken large; the proof below specifies the mild containment requirement on \(K\).

The last ingredient is an order penalty.  If no penalty is used, adding a lag can only improve, or at least not worsen, the least-squares fit; for instance,
\begin{equation}
    \widehat\sigma_n^2(p+1,q)\le \widehat\sigma_n^2(p,q),
    \qquad
    \widehat\sigma_n^2(p,q+1)\le \widehat\sigma_n^2(p,q),
\end{equation}
whenever the two candidate orders are inside the search rectangle.  Thus minimizing the residual variance alone would tend to favor unnecessarily large orders.  We therefore choose a deterministic sequence \(\pi_n>0\) before the minimization is carried out, as in generalized information criteria of Schwarz and Hannan--Quinn type \cite{Schwarz1978,HannanQuinn1979}.  It charges one unit of complexity for each AR or MA coefficient:
\begin{equation}
    \text{fit term}=\log \widehat\sigma_n^2(p,q),
    \qquad
    \text{complexity cost}=(p+q)\pi_n .
\end{equation}
Equivalently, adding one more AR lag is accepted only if the reduction in log residual variance is larger than the penalty, since
\begin{equation}
    \mathrm{HIC}_n(p+1,q)<\mathrm{HIC}_n(p,q)
    \quad\Longrightarrow\quad
    \log\frac{\widehat\sigma_n^2(p+1,q)}{\widehat\sigma_n^2(p,q)}<-\pi_n,
\end{equation}
and the same interpretation applies to adding one MA lag.  Thus a larger \(\pi_n\) gives a more conservative selector, whereas a smaller \(\pi_n\) allows more parameters.  The asymptotic penalty window used in the proof is stated later in Section \ref{sec:proof-main}; in finite-sample computation one may use a BIC-scale choice such as \(c\log n/n\), as in the simulation study.

With this penalty, define
\begin{equation}
    \mathrm{HIC}_n(p,q)=\log \widehat\sigma_n^2(p,q)+(p+q)\pi_n .
\end{equation}
The proposed order selector is
\begin{equation}\label{eq:selector}
    (\widehat p_n,\widehat q_n)
    \in\argmin_{0\le p\le P_n,\,0\le q\le Q_n}\mathrm{HIC}_n(p,q).
\end{equation}
Thus the computing procedure is: estimate \(d_0\), filter the data, construct Hannan--Rissanen residuals, compute the residual variance for each \((p,q)\), and minimize the penalized criterion over the growing rectangle.

\subsection{Main theorem}

The theorem states that the selector defined in \eqref{eq:selector} consistently recovers the finite AR and MA orders in Assumption \ref{ass:truth}.  The detailed approximation, separation and penalty conditions are introduced in Sections \ref{sec:approximation} and \ref{sec:proof-main}, where they are used in the proof.
\begin{theorem}\label{thm:order-consistency}
Under Assumptions \ref{ass:truth} and \ref{ass:hn-growth}, under the preliminary rate condition in Lemma \ref{lem:gph-rate}, and under the separation and penalty-window conditions specified in Section \ref{sec:proof-main}, the selector \eqref{eq:selector} satisfies
\begin{equation}\label{eq:main-consistency}
    \Pp\{(\widehat p_n,\widehat q_n)=(p_0,q_0)\}\longrightarrow1.
\end{equation}
\end{theorem}

\section{Simulation study}\label{sec:simulation}

We include a small Monte Carlo check that follows the theoretical construction in the same order as the method: generate a finite-order ARFIMA series, estimate the memory parameter, fractionally filter the data, construct Hannan--Rissanen residuals, and then minimize the generalized information criterion over a growing candidate rectangle.  The purpose is not to tune the finite-sample implementation exhaustively, but to illustrate the algorithmic sequence used in the consistency theorem.

The innovations are standard normal.  For each design we generate
\[
    \Phi_0(B)(1-B)^{d_0}Y_t=\Theta_0(B)\varepsilon_t,
\]
apply the GPH estimator with bandwidth \(m=\lfloor n^{0.65}\rfloor\), construct \(\widetilde X_t=(1-B)^{\widetilde d}Y_t\), compute Hannan--Rissanen residuals using
\[
    h=\max\{30,\lfloor 3\log n\rfloor\},
\]
and minimize \(\mathrm{HIC}_n(p,q)\) over the growing rectangle
\[
    0\le p\le P_n,
    \qquad
    0\le q\le Q_n,
    \qquad
    P_n=Q_n=\lfloor 1.25\log n\rfloor.
\]
Thus the search ranges are \(10,11,12\) for \(n=4096,8192,16384\), respectively.  For this finite-sample diagnostic we use the BIC-scale generalized penalty
\[
    \pi_n=3\frac{\log n}{n}.
\]
The theoretical result allows larger penalties when the preliminary filtering and residual-construction errors require them.  The present simulation is meant to mirror the steps of the theorem rather than to verify every asymptotic rate restriction numerically.  The compact coefficient ball in \eqref{eq:K-ball} was taken large enough that the numerical solution coincides with the usual least-squares HR residual variance in these designs.

\begin{table}[H]
\centering
\caption{Monte Carlo designs.  The AR coefficient vector is \((\alpha_{0,1},\ldots,\alpha_{0,p_0})\), and the MA coefficient vector is \((\beta_{0,1},\ldots,\beta_{0,q_0})\).}

\small
\begin{tabular}{c c c c c}
\hline
Model & \(d_0\) & AR coefficients & MA coefficients & \((p_0,q_0)\) \\
\hline
M1 & 0.30 & \((0.40,-0.30,0.35)\) & \((-0.076,0.086,-0.377,0.434)\) & \((3,4)\) \\
M2 & 0.25 & \((0.50,-0.40,0.30,-0.35)\) & \((0.172,-0.300,0.442,-0.135,-0.395)\) & \((4,5)\) \\
M3 & 0.35 & \((0.60,-0.50)\) & \((-0.200,0.450,-0.550)\) & \((2,3)\) \\
\hline
\end{tabular}
\end{table}

Table \ref{tab:simulation-results} reports 30 independent replications for each \((\text{model},n)\) pair.  The column ``correct'' records the number of replications in which the selected order equals the true order.  The last column gives the empirical selection frequencies.

\begin{table}[H]
\centering
\caption{Order selection frequencies for the GPH--HR--GIC selector with growing candidate rectangles.}
\label{tab:simulation-results}
\small
\begin{tabular}{c r c c c c c}
\hline
Model & \(n\) & \((P_n,Q_n)\) & \((p_0,q_0)\) & correct & rate & selected orders \\
\hline
M1 & 4096 & \((10,10)\) & \((3,4)\) & 30/30 & 1.000 & \((3,4):30\) \\
M1 & 8192 & \((11,11)\) & \((3,4)\) & 30/30 & 1.000 & \((3,4):30\) \\
M1 & 16384 & \((12,12)\) & \((3,4)\) & 30/30 & 1.000 & \((3,4):30\) \\
M2 & 4096 & \((10,10)\) & \((4,5)\) & 30/30 & 1.000 & \((4,5):30\) \\
M2 & 8192 & \((11,11)\) & \((4,5)\) & 30/30 & 1.000 & \((4,5):30\) \\
M2 & 16384 & \((12,12)\) & \((4,5)\) & 30/30 & 1.000 & \((4,5):30\) \\
M3 & 4096 & \((10,10)\) & \((2,3)\) & 29/30 & 0.967 & \((2,3):29;(3,4):1\) \\
M3 & 8192 & \((11,11)\) & \((2,3)\) & 30/30 & 1.000 & \((2,3):30\) \\
M3 & 16384 & \((12,12)\) & \((2,3)\) & 30/30 & 1.000 & \((2,3):30\) \\
\hline
\end{tabular}
\end{table}

The experiment follows the same operational order as the theory.  The candidate rectangle grows from \((10,10)\) to \((12,12)\), and the selected order concentrates at the true higher-order pair in all three designs.  The only error is a single one-step overfit for M3 at \(n=4096\).  No zero-order design is used in this experiment.

\section{Approximation theory}\label{sec:approximation}

The purpose of this section is to prove that the two feasible objects used by the selector, \(\widetilde X_t\) and \(\widetilde e_t\), are close to their ideal counterparts.  All empirical sums in the proof sections are taken over \(\mathcal I_n\).  For a finite sequence \((u_t)_{t\in\mathcal I_n}\), write
\begin{equation}
    \|u\|_N=\left(N^{-1}\sum_{t\in\mathcal I_n}u_t^2\right)^{1/2}.
\end{equation}

\subsection{Finite-section concentration}

We first record the concentration input needed to control high-order regressions and uniform empirical losses.
\begin{lemma}\label{lem:finite-section-concentration}
Suppose Assumption \ref{ass:truth} holds.  Let \(b_n\ge1\) be deterministic, and let \(W_{t,n}\in\R^{d_n}\) be any vector whose coordinates are chosen from finitely many lags of \((X_t)\) and \((\varepsilon_t)\), with \(d_n\le C b_n\).  If \(b_n\log n/N\to0\), then
\begin{equation}\label{eq:finite-section-covariance}
    \left\|
    N^{-1}\sum_t W_{t,n}W_{t,n}^\top-
    \E W_{t,n}W_{t,n}^\top
    \right\|_{\op}
    =O_p\left\{\left(\frac{b_n\log n}{N}\right)^{1/2}\right\}.
\end{equation}
Moreover, if there are constants \(0<c_\Sigma\le C_\Sigma<\infty\) such that
\begin{equation}\label{eq:population-eigenvalue-window}
    c_\Sigma
    \le
    \lambda_{\min}\{\E W_{t,n}W_{t,n}^{\top}\}
    \le
    \lambda_{\max}\{\E W_{t,n}W_{t,n}^{\top}\}
    \le C_\Sigma,
\end{equation}
then
\begin{equation}\label{eq:empirical-eigenvalue-window}
    \Pp\left\{
    \frac{c_\Sigma}{2}
    \le
    \lambda_{\min}\left(N^{-1}\sum_t W_{t,n}W_{t,n}^{\top}\right)
    \le
    \lambda_{\max}\left(N^{-1}\sum_t W_{t,n}W_{t,n}^{\top}\right)
    \le
    2C_\Sigma
    \right\}\longrightarrow1.
\end{equation}
Finally, for \(X_{t,h}=(X_{t-1},\ldots,X_{t-h})^\top\),
\begin{equation}\label{eq:finite-section-score}
    \left\|N^{-1}\sum_t X_{t,h}\varepsilon_t\right\|_2
    =O_p\left\{\left(\frac{h\log n}{n}\right)^{1/2}\right\}
\end{equation}
whenever \(h\log n/n\to0\).
\end{lemma}

\begin{proof}
Put
\begin{equation}
    A_n=N^{-1}\sum_t W_{t,n}W_{t,n}^{\top}-\E W_{t,n}W_{t,n}^{\top}.
\end{equation}
By \eqref{eq:causal-representation}, each coordinate of \(W_{t,n}\) is a geometrically stable linear process driven by independent sub-Gaussian innovations.  For a unit vector \(u\in\R^{d_n}\), define
\[
    Q_t(u)=(u^\top W_{t,n})^2-\E (u^\top W_{t,n})^2 .
\]
The process \((Q_t(u))\) is centered, sub-exponential, and geometrically stable, uniformly over all \(u\) satisfying \(\|u\|_2=1\).  Bernstein-type inequalities for weakly dependent and physical-dependence time series \cite{Wu2005,MerlevedePeligradRio2011,BasuMichailidis2015,ZhangWu2017} therefore give constants \(c_0,C_0>0\), independent of \(n\), such that for every \(u\in\R^{d_n}\) with \(\|u\|_2=1\) and every \(x>0\),
\begin{equation}\label{eq:one-direction-bernstein}
    \Pp\left(
    \left|u^{\top}A_nu\right|>x
    \right)
    \le
    C_0\exp\{-c_0N\min(x^2,x)\}.
\end{equation}
Let \({\mathcal N}_n\) be a \(1/4\)-net of the unit sphere
\[
    \{u\in\R^{d_n}:\|u\|_2=1\}.
\]
Then
\begin{equation}
    |{\mathcal N}_n|\le 9^{d_n}
    \le \exp(C_1b_n),
\end{equation}
and the standard net inequality for symmetric matrices gives
\begin{equation}
    \|A_n\|_{\op}
    \le 2\max_{u\in {\mathcal N}_n}|u^{\top}A_nu|.
\end{equation}
Set
\begin{equation}\label{eq:cov-xn-choice}
    x_n=L\left(\frac{b_n\log n}{N}\right)^{1/2}.
\end{equation}
Since \(b_n\log n/N\to0\), \(x_n=o(1)\).  Combining \eqref{eq:one-direction-bernstein}--\eqref{eq:cov-xn-choice}, for all large \(n\),
\begin{align}
    \Pp\{\|A_n\|_{\op}>2x_n\}
    &\le
    C_0|{\mathcal N}_n|\exp(-c_0Nx_n^2)  \\
    &\le
    C_0\exp\{C_1b_n-c_0L^2b_n\log n\}.
\end{align}
Choosing \(L\) fixed and large enough yields
\begin{equation}\label{eq:cov-op-rate-proof}
    \|A_n\|_{\op}
    =O_p\left\{\left(\frac{b_n\log n}{N}\right)^{1/2}\right\},
\end{equation}
which is \eqref{eq:finite-section-covariance}.

Assume \eqref{eq:population-eigenvalue-window}.  On the event \(\|A_n\|_{\op}\le c_\Sigma/2\), Weyl's inequality gives
\begin{equation}
    c_\Sigma/2
    \le
    \lambda_{\min}\left(N^{-1}\sum_t W_{t,n}W_{t,n}^{\top}\right)
    \le
    \lambda_{\max}\left(N^{-1}\sum_t W_{t,n}W_{t,n}^{\top}\right)
    \le C_\Sigma+c_\Sigma/2
    \le 2C_\Sigma.
\end{equation}
Since \(\Pp\{\|A_n\|_{\op}\le c_\Sigma/2\}\to1\) by \eqref{eq:cov-op-rate-proof}, \eqref{eq:empirical-eigenvalue-window} follows.

It remains to prove \eqref{eq:finite-section-score}.  For \(1\le j\le h\), set
\begin{equation}
    S_{n,j}=N^{-1}\sum_t X_{t-j}\varepsilon_t.
\end{equation}
Since \(X_{t-j}\) is measurable with respect to \(\sigma(\varepsilon_s:s\le t-j)\) and is independent of \(\varepsilon_t\), the products \(X_{t-j}\varepsilon_t\) are centered sub-exponential stable time-series variables.  The same Bernstein inequality yields
\begin{equation}
    \Pp(|S_{n,j}|>x)
    \le C_0\exp\{-c_0N\min(x^2,x)\}.
\end{equation}
With
\begin{equation}
    y_n=L\left(\frac{\log n}{N}\right)^{1/2},
\end{equation}
we have, because \(h\log n/n\to0\) and \(N\asymp n\),
\begin{align}
    \Pp\left(\max_{1\le j\le h}|S_{n,j}|>y_n\right)
    &\le C_0h\exp(-c_0Ny_n^2) \\
    &\le C_0h n^{-c_0L^2}.
\end{align}
Taking \(L\) large gives
\begin{equation}
    \max_{1\le j\le h}|S_{n,j}|=O_p\left\{\left(\frac{\log n}{N}\right)^{1/2}\right\}.
\end{equation}
Therefore
\begin{equation}
    \left\|N^{-1}\sum_t X_{t,h}\varepsilon_t\right\|_2
    =\left(\sum_{j=1}^{h}S_{n,j}^2\right)^{1/2}
    \le h^{1/2}\max_{1\le j\le h}|S_{n,j}|
    =O_p\left\{\left(\frac{h\log n}{n}\right)^{1/2}\right\}.
\end{equation}
This proves \eqref{eq:finite-section-score}.
\end{proof}

\subsection{Fractional-filtering perturbation}

The next lemma controls the difference between the feasible filtered series \(\widetilde X_t\) and the infeasible short-memory series \(X_t\).
\begin{lemma}\label{lem:filtering-perturbation}
Under Assumption \ref{ass:truth} and Lemma \ref{lem:gph-rate},
\begin{equation}\label{eq:filtering-perturbation}
    \|\widetilde X-X\|_N=O_p(a_{d,n}\log n).
\end{equation}
\end{lemma}

\begin{proof}
Put \(\delta_n=\widetilde d-d_0\).  On the event \(|\delta_n|\le\Delta_0\), with \(\Delta_0>0\) fixed sufficiently small,
\begin{equation}\label{eq:fractional-ratio}
    \widetilde X_t-X_t=\bigl\{(1-B)^{\delta_n}-1\bigr\}X_t
    =\delta_n\sum_{\ell=1}^{t-1}\kappa_\ell(\bar\delta_n)X_{t-\ell},
\end{equation}
where \(|\bar\delta_n|\le\Delta_0\) and
\begin{equation}\label{eq:kappa-bound}
    \sup_{|\delta|\le\Delta_0}|\kappa_\ell(\delta)|\le C\ell^{-1},
    \qquad 1\le \ell\le n.
\end{equation}
For \(|\delta|\le\Delta_0\), define
\begin{equation}
    U_t(\delta)=\sum_{\ell=1}^{t-1}\kappa_\ell(\delta)X_{t-\ell}.
\end{equation}
By \eqref{eq:kappa-bound},
\begin{equation}
    |U_t(\delta)|
    \le
    C\sum_{\ell=1}^{n}\ell^{-1}|X_{t-\ell}|,
    \qquad |\delta|\le\Delta_0 .
\end{equation}
Since \((X_t)\) is stationary with \(\sup_t\E X_t^2<\infty\), Cauchy's inequality gives
\begin{align}
    \E\left[\sup_{|\delta|\le\Delta_0}\|U(\delta)\|_N^2\right]
    &\le
    C N^{-1}\sum_{t\in\mathcal I_n}
    \E\left(\sum_{\ell=1}^{n}\ell^{-1}|X_{t-\ell}|\right)^2 \\
    &\le
    C\left(\sum_{\ell=1}^{n}\ell^{-1}\right)^2
    =O(\log^2 n).
\end{align}
Hence, by Markov's inequality,
\begin{equation}\label{eq:filtering-uniform-bound}
    \sup_{|\delta|\le\Delta_0}\left\|
    \sum_{\ell=1}^{t-1}\kappa_\ell(\delta)X_{t-\ell}
    \right\|_N
    =O_p(\log n).
\end{equation}
Combining \eqref{eq:fractional-ratio}, \eqref{eq:filtering-uniform-bound}, and \(|\delta_n|=O_p(a_{d,n})\) gives \eqref{eq:filtering-perturbation}.
\end{proof}

\subsection{Hannan--Rissanen residual approximation}

This subsection proves that the residuals used in the feasible ARMA-type regressions are close to the unobserved innovations.  The only extra ingredient is the truncation error of the infinite autoregression in \eqref{eq:infinite-ar-representation}.  For each integer \(h\ge1\), set
\begin{equation}\label{eq:tau-eta-def}
    \tau_{h,t}=\sum_{\ell>h}\varphi_{0,\ell}X_{t-\ell},
    \qquad
    \eta_h=(\E\tau_{h,t}^2)^{1/2}.
\end{equation}
Since \(\sum_{\ell\ge1}|\varphi_{0,\ell}|R^\ell<\infty\) for some \(R>1\) and \(\E X_t^2<\infty\), Minkowski's inequality gives, for some constants \(C<\infty\) and \(R_1>1\),
\begin{equation}
    \eta_h
    \le (\E X_t^2)^{1/2}\sum_{\ell>h}|\varphi_{0,\ell}|
    \le C R_1^{-h}.
\end{equation}
\begin{assumption}\label{ass:hn-growth}
The sequence \(h_n\) satisfies
\begin{equation}\label{eq:hn-growth}
    h_na_{d,n}\log n\to0,
    \qquad
    h_n^{1/2}\eta_{h_n}\to0,
    \qquad
    h_n^2\log n/n\to0.
\end{equation}
\end{assumption}
Define the residual-approximation scale
\begin{equation}\label{eq:re-rate}
    r_{e,n}
    =h_n^{1/2}a_{d,n}\log n+
    \eta_{h_n}+
    (h_n\log n/n)^{1/2}.
\end{equation}
\begin{lemma}\label{lem:hr-residual-approx}
Under Assumption \ref{ass:truth}, Lemma \ref{lem:gph-rate}, and Assumption \ref{ass:hn-growth},
\begin{equation}\label{eq:hr-residual-approx}
    \|\widetilde e-\varepsilon\|_N=O_p(r_{e,n}),
    \qquad
    \|\widehat\varphi\|_1=O_p(1).
\end{equation}
\end{lemma}

\begin{proof}
Write \(h=h_n\).  Define
\begin{equation}
    X_{t,h}=(X_{t-1},\ldots,X_{t-h})^\top,
    \qquad
    \widetilde X_{t,h}=(\widetilde X_{t-1},\ldots,\widetilde X_{t-h})^\top,
\end{equation}
and let
\begin{equation}
    \varphi_{0,h}=(\varphi_{0,1},\ldots,\varphi_{0,h})^\top.
\end{equation}
By \eqref{eq:infinite-ar-representation} and \eqref{eq:tau-eta-def},
\begin{equation}\label{eq:true-high-ar}
    X_t=X_{t,h}^\top\varphi_{0,h}+\tau_{h,t}+\varepsilon_t.
\end{equation}
Moreover,
\begin{equation}\label{eq:tau-empirical-bound}
    \E\|\tau_h\|_N^2
    =N^{-1}\sum_{t\in\mathcal I_n}\E\tau_{h,t}^2
    =\eta_h^2,
    \qquad
    \|\tau_h\|_N=O_p(\eta_h).
\end{equation}

Define
\begin{equation}\label{eq:gram-defs}
    \widehat\Gamma_h=N^{-1}\sum_t X_{t,h}X_{t,h}^{\top},
    \qquad
    \widetilde\Gamma_h=N^{-1}\sum_t \widetilde X_{t,h}\widetilde X_{t,h}^{\top}.
\end{equation}
The population covariance of \(X_{t,h}\) has eigenvalues bounded away from zero and infinity uniformly in \(h\), because \((X_t)\) is causal and invertible with a spectral density bounded above and below on \([-\pi,\pi]\).  Applying Lemma \ref{lem:finite-section-concentration} with \(W_{t,n}=X_{t,h}\) and \(b_n=h\) gives
\begin{equation}
    \Pp\left\{
    0<c\le \lambda_{\min}(\widehat\Gamma_h)
    \le \lambda_{\max}(\widehat\Gamma_h)
    \le C
    \right\}\to1.
\end{equation}
The score part of Lemma \ref{lem:finite-section-concentration} gives
\begin{equation}
    \left\|N^{-1}\sum_t X_{t,h}\varepsilon_t\right\|_2
    =O_p\{(h\log n/n)^{1/2}\}.
\end{equation}

Put
\begin{equation}
    \Delta X_t=\widetilde X_t-X_t,
    \qquad
    \Delta X_{t,h}=\widetilde X_{t,h}-X_{t,h}.
\end{equation}
For \(0\le j\le h\), apply the argument of Lemma \ref{lem:filtering-perturbation} to the shifted sample \(\{t-j:t\in\mathcal I_n\}\).  Since \(h=o(n)\) by \eqref{eq:hn-growth}, this gives the uniform lag bound
\begin{equation}
    \max_{0\le j\le h}\|\Delta X_{\cdot-j}\|_N
    =O_p(a_{d,n}\log n).
\end{equation}
Consequently,
\begin{equation}
    \|\Delta X\|_N=O_p(a_{d,n}\log n),
    \qquad
    \left(N^{-1}\sum_t\|\Delta X_{t,h}\|_2^2\right)^{1/2}
    =O_p(h^{1/2}a_{d,n}\log n).
\end{equation}
Using \eqref{eq:gram-defs},
\begin{align}
    \|\widetilde\Gamma_h-\widehat\Gamma_h\|_{\op}
    &\le
    2\left\|N^{-1}\sum_t X_{t,h}\Delta X_{t,h}^{\top}\right\|_{\op}
    +\left\|N^{-1}\sum_t\Delta X_{t,h}\Delta X_{t,h}^{\top}\right\|_{\op} \\
    &\le
    2\lambda_{\max}(\widehat\Gamma_h)^{1/2}
    \left(N^{-1}\sum_t\|\Delta X_{t,h}\|_2^2\right)^{1/2}
    +N^{-1}\sum_t\|\Delta X_{t,h}\|_2^2 \\
    &=O_p(h^{1/2}a_{d,n}\log n+h a_{d,n}^2\log^2 n)
    =o_p(1),
\end{align}
where the last step follows from \eqref{eq:hn-growth}.  Therefore
\begin{equation}\label{eq:tilde-gram-inverse-bound}
    \|\widetilde\Gamma_h^{-1}\|_{\op}=O_p(1),
    \qquad
    \lambda_{\max}(\widetilde\Gamma_h)=O_p(1).
\end{equation}

The least-squares normal equation for \(\widehat\varphi\) is
\begin{equation}\label{eq:feasible-normal-equation-exact}
    \widehat\varphi-\varphi_{0,h}
    =\widetilde\Gamma_h^{-1}
    N^{-1}\sum_t \widetilde X_{t,h}
    \{\widetilde X_t-\widetilde X_{t,h}^\top\varphi_{0,h}\}.
\end{equation}
By \eqref{eq:true-high-ar},
\begin{equation}\label{eq:feasible-normal-residual-decomposition}
    \widetilde X_t-\widetilde X_{t,h}^\top\varphi_{0,h}
    =\varepsilon_t+	au_{h,t}
     +\Delta X_t-\Delta X_{t,h}^\top\varphi_{0,h}.
\end{equation}
We bound the four summands in \eqref{eq:feasible-normal-residual-decomposition}.  First,
\begin{align}
    \left\|N^{-1}\sum_t \widetilde X_{t,h}\varepsilon_t\right\|_2
    &\le
    \left\|N^{-1}\sum_t X_{t,h}\varepsilon_t\right\|_2
    +\left\|N^{-1}\sum_t \Delta X_{t,h}\varepsilon_t\right\|_2 \\
    &\le
    O_p\{(h\log n/n)^{1/2}\}
    +
    \left(N^{-1}\sum_t\|\Delta X_{t,h}\|_2^2\right)^{1/2}
    \left(N^{-1}\sum_t\varepsilon_t^2\right)^{1/2} \\
    &=O_p\{(h\log n/n)^{1/2}\}+O_p(h^{1/2}a_{d,n}\log n).
\end{align}
Second, by Cauchy's inequality, \eqref{eq:tau-empirical-bound}, and \eqref{eq:tilde-gram-inverse-bound},
\begin{equation}
    \left\|N^{-1}\sum_t \widetilde X_{t,h}\tau_{h,t}\right\|_2
    \le \lambda_{\max}(\widetilde\Gamma_h)^{1/2}\|\tau_h\|_N
    =O_p(\eta_h).
\end{equation}
Third,
\begin{equation}
    \left\|N^{-1}\sum_t \widetilde X_{t,h}\Delta X_t\right\|_2
    \le \lambda_{\max}(\widetilde\Gamma_h)^{1/2}\|\Delta X\|_N
    =O_p(a_{d,n}\log n).
\end{equation}
Finally, since \(\sum_{\ell\ge1}|\varphi_{0,\ell}|<\infty\),
\begin{equation}
    \|\varphi_{0,h}\|_1=O(1),
    \qquad
    \|\varphi_{0,h}\|_2=O(1).
\end{equation}
Therefore,
\begin{align}
    \left\|N^{-1}\sum_t \widetilde X_{t,h}\Delta X_{t,h}^\top\varphi_{0,h}\right\|_2
    &\le \lambda_{\max}(\widetilde\Gamma_h)^{1/2}
       \left(N^{-1}\sum_t(\Delta X_{t,h}^\top\varphi_{0,h})^2\right)^{1/2} \\
    &\le O_p(1)\,\|\varphi_{0,h}\|_2
       \left(N^{-1}\sum_t\|\Delta X_{t,h}\|_2^2\right)^{1/2} \\
    &=O_p(h^{1/2}a_{d,n}\log n).
\end{align}
Combining the normal equation \eqref{eq:feasible-normal-equation-exact}, the decomposition \eqref{eq:feasible-normal-residual-decomposition}, and the four bounds above yields
\begin{equation}
    \|\widehat\varphi-\varphi_{0,h}\|_2
    =O_p\left(h^{1/2}a_{d,n}\log n+
    \eta_h+(h\log n/n)^{1/2}\right)
    =O_p(r_{e,n}).
\end{equation}
Furthermore, by \eqref{eq:hn-growth},
\begin{align}
    \|\widehat\varphi\|_1
    &\le \|\varphi_{0,h}\|_1+h^{1/2}\|\widehat\varphi-\varphi_{0,h}\|_2 \\
    &=O_p(1)+O_p\{h a_{d,n}\log n+h^{1/2}\eta_h+h(\log n/n)^{1/2}\}
    =O_p(1).
\end{align}

It remains to compare the feasible residuals with \(\varepsilon_t\).  From \eqref{eq:hr-residual} and \eqref{eq:true-high-ar},
\begin{equation}\label{eq:residual-error-decomp}
    \widetilde e_t-\varepsilon_t
    =\Delta X_t
     -\widehat\varphi^\top\Delta X_{t,h}
     -(\widehat\varphi-\varphi_{0,h})^\top X_{t,h}
     +\tau_{h,t}.
\end{equation}
The four terms in \eqref{eq:residual-error-decomp} satisfy
\begin{align}
    \|\Delta X\|_N
    &=O_p(a_{d,n}\log n), \\
    \left\|\widehat\varphi^\top\Delta X_{\cdot,h}\right\|_N
    &\le \|\widehat\varphi\|_2
    \left(N^{-1}\sum_t\|\Delta X_{t,h}\|_2^2\right)^{1/2}
    =O_p(h^{1/2}a_{d,n}\log n), \\
    \left\|(\widehat\varphi-\varphi_{0,h})^\top X_{\cdot,h}\right\|_N
    &\le \lambda_{\max}(\widehat\Gamma_h)^{1/2}
    \|\widehat\varphi-\varphi_{0,h}\|_2
    =O_p(r_{e,n}), \\
    \|\tau_h\|_N&=O_p(\eta_h).
\end{align}
Together with the definition \eqref{eq:re-rate}, these bounds imply
\begin{equation}
    \|\widetilde e-\varepsilon\|_N=O_p(r_{e,n}).
\end{equation}
This proves \eqref{eq:hr-residual-approx}.
\end{proof}

\section{Proof of the main theorem}\label{sec:proof-main}

The proof compares the computable residual variances in Section \ref{sec:method-main} with their ideal ARMA counterparts.  For a candidate order \((p,q)\), define
\begin{equation}
    Z_t^0(p,q)=
    (X_{t-1},\ldots,X_{t-p},
    -\varepsilon_{t-1},\ldots,-\varepsilon_{t-q})^\top
\end{equation}
and
\begin{equation}
    L^0_{p,q}(\gamma)=\E\{X_t-Z_t^0(p,q)^\top\gamma\}^2,
    \qquad
    \sigma_{K,p,q}^{2,*}=\inf_{\gamma\in\mathcal K_{p,q}(K)}L^0_{p,q}(\gamma).
\end{equation}
Choose the compact radius \(K\) so that the zero-padded true coefficient vector is in \(\mathcal K_{p,q}(K)\) whenever \(p\ge p_0\) and \(q\ge q_0\).  Equivalently, it is enough to take
\begin{equation}
    K>
    \left(
    \sum_{j=1}^{p_0}\alpha_{0,j}^2+
    \sum_{k=1}^{q_0}\beta_{0,k}^2
    \right)^{1/2} .
\end{equation}
Then
\begin{equation}\label{eq:population-overfit}
    \sigma_{K,p,q}^{2,*}=\sigma_\varepsilon^2,
    \qquad
    p\ge p_0,
    \quad
    q\ge q_0.
\end{equation}
The underfitted candidate set is
\begin{equation}
    \mathcal U_n=
    \{(p,q):0\le p\le P_n,\ 0\le q\le Q_n,\ p<p_0\text{ or }q<q_0\}.
\end{equation}
Define
\begin{equation}
    \Delta_n=
    \begin{cases}
    \displaystyle
    \inf_{(p,q)\in\mathcal U_n}
    \left\{
    \log\sigma_{K,p,q}^{2,*}-\log\sigma_\varepsilon^2
    \right\}, & \mathcal U_n\ne\varnothing,\\[1.2em]
    +\infty, & \mathcal U_n=\varnothing.
    \end{cases}
\end{equation}
The uniform approximation scale is
\begin{equation}\label{eq:omega-n}
    \omega_n=
    \left(\frac{M_n\log n}{N}\right)^{1/2}
    +M_n^{1/2}r_{e,n}
    +M_n r_{e,n}^2 .
\end{equation}
The penalty sequence is required to satisfy
\begin{equation}\label{eq:sieve-penalty-window}
    \frac{\omega_n}{\pi_n}\longrightarrow0,
    \qquad
    \frac{M_n\pi_n}{\Delta_n}\longrightarrow0,
    \qquad
    \pi_n\longrightarrow0,
\end{equation}
where the second ratio is interpreted as zero when \(\mathcal U_n=\varnothing\).  Throughout the final comparison, the minimum over an empty set is interpreted as \(+\infty\).  Thus
\begin{equation}
    \omega_n\to0,
    \qquad
    \frac{M_n\log n}{N}\to0,
    \qquad
    \omega_n\ll \pi_n\ll \frac{\Delta_n}{M_n}.
\end{equation}

\subsection{Population separation and uniform residual-variance approximation}

The first two lemmas identify the population ARMA benchmark.
\begin{lemma}
Under Assumption \ref{ass:truth}, for every fixed finite pair \((p,q)\),
\begin{equation}\label{eq:s5-fixed-underfit-gap}
    p<p_0\text{ or }q<q_0
    \quad\Longrightarrow\quad
    \sigma_{K,p,q}^{2,*}>\sigma_\varepsilon^2.
\end{equation}
Consequently, for fixed finite \(P,Q\), if
\begin{equation}
    \mathcal U(P,Q)=
    \{(p,q):0\le p\le P,\ 0\le q\le Q,
    \ p<p_0\text{ or }q<q_0\}
    \ne\varnothing,
\end{equation}
then
\begin{equation}\label{eq:s5-fixed-rectangle-gap}
    \inf_{(p,q)\in\mathcal U(P,Q)}
    \left\{\log\sigma_{K,p,q}^{2,*}-\log\sigma_\varepsilon^2\right\}>0 .
\end{equation}
\end{lemma}
\begin{proof}
Fix \((p,q)\) with \(p<p_0\) or \(q<q_0\).  For \(\gamma\in\mathcal K_{p,q}(K)\), set
\begin{equation}
    R_t(\gamma)=X_t-Z_t^0(p,q)^\top\gamma,
    \qquad
    D_t(\gamma)=X_t-\varepsilon_t-Z_t^0(p,q)^\top\gamma .
\end{equation}
Since \(D_t(\gamma)\) is measurable with respect to
\(\mathcal F_{t-1}=\sigma(\varepsilon_{t-1},\varepsilon_{t-2},\ldots)\),
\begin{equation}
    R_t(\gamma)=\varepsilon_t+D_t(\gamma),
    \qquad
    \E\{\varepsilon_tD_t(\gamma)\}=0,
\end{equation}
and therefore
\begin{equation}\label{eq:s5-loss-decomposition}
    L_{p,q}^0(\gamma)=\sigma_\varepsilon^2+\E D_t(\gamma)^2
    \ge \sigma_\varepsilon^2.
\end{equation}
The minimum of \(L_{p,q}^0\) over \(\mathcal K_{p,q}(K)\) is attained, because the loss is continuous and the ball is compact.  If equality held in \eqref{eq:s5-loss-decomposition} at a minimizer \(\gamma_*\), then
\begin{equation}\label{eq:s5-dzero}
    \E D_t(\gamma_*)^2=0,
    \qquad
    R_t(\gamma_*)=\varepsilon_t\quad\text{in }L^2.
\end{equation}
Write
\begin{equation}
    \gamma_*=(a_1,\ldots,a_p,c_1,\ldots,c_q)^\top.
\end{equation}
Then \eqref{eq:s5-dzero} is equivalent to
\begin{equation}
    X_t-\sum_{j=1}^{p}a_jX_{t-j}+
    \sum_{k=1}^{q}c_k\varepsilon_{t-k}=\varepsilon_t .
\end{equation}
Define
\begin{equation}
    \Phi(z)=1-\sum_{j=1}^{p}a_jz^j,
    \qquad
    \Theta(z)=1-\sum_{k=1}^{q}c_kz^k .
\end{equation}
Then
\begin{equation}\label{eq:s5-alt-arma-rep}
    \Phi(B)X_t=\Theta(B)\varepsilon_t,
    \qquad
    \deg\Phi\le p,
    \quad
    \deg\Theta\le q,
    \quad
    \Phi(0)=\Theta(0)=1.
\end{equation}
Combining \eqref{eq:s5-alt-arma-rep} with \(\Phi_0(B)X_t=\Theta_0(B)\varepsilon_t\) gives
\begin{equation}
    \{\Phi(B)\Theta_0(B)-\Theta(B)\Phi_0(B)\}\varepsilon_t=0 .
\end{equation}
If \(A(z)=\Phi(z)\Theta_0(z)-\Theta(z)\Phi_0(z)=\sum_{\ell=0}^{m}A_\ell z^\ell\), then independence of \((\varepsilon_t)\) gives
\begin{equation}
    0=
    \E\left(\sum_{\ell=0}^{m}A_\ell\varepsilon_{t-\ell}\right)^2
    =\sigma_\varepsilon^2\sum_{\ell=0}^{m}A_\ell^2 .
\end{equation}
Hence \(A_\ell=0\) for all \(\ell\), namely
\begin{equation}
    \Phi(z)\Theta_0(z)=\Theta(z)\Phi_0(z).
\end{equation}
Since \(\Phi_0\) and \(\Theta_0\) are coprime,
\begin{equation}
    \Phi(z)=C(z)\Phi_0(z),
    \qquad
    \Theta(z)=C(z)\Theta_0(z),
    \qquad
    C(0)=1,
\end{equation}
for some polynomial \(C\).  The terminal coefficients in Assumption \ref{ass:truth} imply
\begin{equation}
    \deg\Phi\ge p_0,
    \qquad
    \deg\Theta\ge q_0,
\end{equation}
which contradicts \(\deg\Phi\le p\), \(\deg\Theta\le q\), and \(p<p_0\) or \(q<q_0\).  Thus \eqref{eq:s5-fixed-underfit-gap} holds.  Since \(\mathcal U(P,Q)\) is finite, \eqref{eq:s5-fixed-rectangle-gap} follows.
\end{proof}

\begin{lemma}\label{lem:pop-cov-bound}
Under Assumption \ref{ass:truth},
\begin{equation}\label{eq:population-gram-bound}
    \sup_{p\ge0,\,q\ge0}\lambda_{\max}(\Sigma_{p,q}^0)<\infty,
    \qquad
    \Sigma_{p,q}^0=\E\{Z_t^0(p,q)Z_t^0(p,q)^\top\}.
\end{equation}
\end{lemma}
\begin{proof}
Let
\begin{equation}
    U_t=(X_t,\varepsilon_t)^\top,
    \qquad
    \Gamma(h)=\E\{U_tU_{t-h}^\top\},\quad h\in\mathbb Z .
\end{equation}
The finite-order causal-invertible ARMA law of \(X_t\) gives
\begin{equation}
    C_\Gamma=\sum_{h\in\mathbb Z}\|\Gamma(h)\|_{\op}<\infty .
\end{equation}
For \(v=(a_1,\ldots,a_p,b_1,\ldots,b_q)^\top\), put
\begin{equation}
    c_r=(a_r,-b_r)^\top,
    \qquad
    a_r=0\ (r>p),
    \qquad
    b_r=0\ (r>q).
\end{equation}
Then
\begin{equation}
    v^\top Z_t^0(p,q)=\sum_{r\ge1}c_r^\top U_{t-r},
\end{equation}
and
\begin{align}
    v^\top\Sigma_{p,q}^0v
    &=\sum_{r\ge1}\sum_{s\ge1}c_r^\top\Gamma(s-r)c_s                                      \\
    &\le \sum_{h\in\mathbb Z}\|\Gamma(h)\|_{\op}
       \sum_{r\ge1}\|c_r\|_2\|c_{r+h}\|_2                                               \\
    &\le C_\Gamma\sum_{r\ge1}\|c_r\|_2^2
     =C_\Gamma\|v\|_2^2,
\end{align}
where \(c_{r+h}=0\) if \(r+h<1\).  Taking the supremum over \(\|v\|_2=1\) proves \eqref{eq:population-gram-bound}.
\end{proof}

We next control the ideal empirical losses uniformly over the growing rectangle.
\begin{lemma}\label{lem:uniform-compact-concentration}
Under Assumption \ref{ass:truth},
\begin{equation}
    \max_{0\le p\le P_n,\,0\le q\le Q_n}
    \sup_{\gamma\in\mathcal K_{p,q}(K)}
    \left|
    \ell^0_{n,p,q}(\gamma)-L^0_{p,q}(\gamma)
    \right|
    =O_p\left\{\left(\frac{M_n\log n}{N}\right)^{1/2}\right\},
\end{equation}
where
\begin{equation}
    \ell^0_{n,p,q}(\gamma)=
    N^{-1}\sum_{t\in\mathcal I_n}\{X_t-Z_t^0(p,q)^\top\gamma\}^2 .
\end{equation}
\end{lemma}
\begin{proof}
Let \(d_{p,q}=p+q\) and define
\begin{equation}
    V_t(p,q)=(X_t,Z_t^0(p,q)^\top)^\top\in\mathbb R^{1+d_{p,q}},
    \qquad
    a(\gamma)=(1,-\gamma^\top)^\top .
\end{equation}
Then
\begin{align}
    \ell^0_{n,p,q}(\gamma)-L^0_{p,q}(\gamma)
    &=a(\gamma)^\top A_n(p,q)a(\gamma),\\
    A_n(p,q)&=N^{-1}\sum_{t\in\mathcal I_n} V_t(p,q)V_t(p,q)^\top
    -\E\{V_t(p,q)V_t(p,q)^\top\},
\end{align}
and
\begin{equation}
    \|a(\gamma)\|_2^2\le 1+K^2=:A_K^2.
\end{equation}
Set
\begin{equation}
    V_{t,n}^{\max}=(X_t,X_{t-1},\ldots,X_{t-P_n},-\varepsilon_{t-1},\ldots,-\varepsilon_{t-Q_n})^\top .
\end{equation}
Since every \(V_t(p,q)\) is a coordinate subvector of \(V_{t,n}^{\max}\), Lemma \ref{lem:finite-section-concentration} with \(b_n=M_n+1\) gives
\begin{equation}
    \max_{p,q}\|A_n(p,q)\|_{\op}
    \le
    \left\|
    N^{-1}\sum_{t\in\mathcal I_n}V_{t,n}^{\max}(V_{t,n}^{\max})^\top
    -\E\{V_{t,n}^{\max}(V_{t,n}^{\max})^\top\}
    \right\|_{\op}
    =O_p\left\{\left(\frac{M_n\log n}{N}\right)^{1/2}\right\}.
\end{equation}
Therefore
\begin{align}
    \max_{p,q}\sup_{\gamma\in\mathcal K_{p,q}(K)}
    |\ell^0_{n,p,q}(\gamma)-L^0_{p,q}(\gamma)|
    &\le A_K^2\max_{p,q}\|A_n(p,q)\|_{\op}\\
    &=O_p\left\{\left(\frac{M_n\log n}{N}\right)^{1/2}\right\}.
\end{align}
\end{proof}

The next estimate transfers the ideal ARMA comparison to the feasible filtered criterion.
\begin{lemma}\label{lem:uniform-resvar}
Under Assumption \ref{ass:truth}, Lemma \ref{lem:gph-rate}, Assumption \ref{ass:hn-growth}, and the penalty window \eqref{eq:sieve-penalty-window},
\begin{equation}\label{eq:uniform-resvar}
    \max_{0\le p\le P_n,\,0\le q\le Q_n}
    \left|
    \widehat\sigma_n^2(p,q)-\sigma_{K,p,q}^{2,*}
    \right|
    =O_p(\omega_n),
\end{equation}
and
\begin{equation}\label{eq:uniform-log-resvar}
    \max_{0\le p\le P_n,\,0\le q\le Q_n}
    \left|
    \log\widehat\sigma_n^2(p,q)-\log\sigma_{K,p,q}^{2,*}
    \right|
    =O_p(\omega_n).
\end{equation}
\end{lemma}
\begin{proof}
Define the ideal empirical restricted variance
\begin{equation}
    \widehat\sigma_{0,n}^2(p,q)=
    \inf_{\gamma\in\mathcal K_{p,q}(K)}\ell^0_{n,p,q}(\gamma).
\end{equation}
By Lemma \ref{lem:uniform-compact-concentration},
\begin{align}\label{eq:s5-ideal-pop-bound}
    \max_{p,q}
    \left|\widehat\sigma_{0,n}^2(p,q)-\sigma_{K,p,q}^{2,*}\right|
    &\le
    \max_{p,q}\sup_{\gamma\in\mathcal K_{p,q}(K)}
    |\ell^0_{n,p,q}(\gamma)-L^0_{p,q}(\gamma)| \\
    &=O_p\left\{\left(\frac{M_n\log n}{N}\right)^{1/2}\right\}.
\end{align}
Put
\begin{equation}
    \Delta X_t=\widetilde X_t-X_t,
    \qquad
    \Delta e_t=\widetilde e_t-\varepsilon_t,
    \qquad
    \Delta Z_t(p,q)=\widetilde Z_t(p,q)-Z_t^0(p,q).
\end{equation}
By Lemmas \ref{lem:filtering-perturbation} and \ref{lem:hr-residual-approx}, applied on the common interior lag range,
\begin{equation}\label{eq:s5-basic-delta-rates}
    \|\Delta X\|_N=O_p(a_{d,n}\log n),
    \qquad
    \|\Delta e\|_N=O_p(r_{e,n}),
    \qquad
    a_{d,n}\log n\le r_{e,n}
\end{equation}
for all large \(n\).  Moreover,
\begin{align}
    \sup_{p,q}N^{-1}\sum_{t\in\mathcal I_n}\|\Delta Z_t(p,q)\|_2^2
    &\le
    M_n\|\Delta X\|_N^2+M_n\|\Delta e\|_N^2 \\
    &=O_p(M_nr_{e,n}^2).
\end{align}
Hence
\begin{equation}\label{eq:s5-deltaZ-norm-rate}
    \sup_{p,q}
    \left(N^{-1}\sum_{t\in\mathcal I_n}\|\Delta Z_t(p,q)\|_2^2\right)^{1/2}
    =O_p(M_n^{1/2}r_{e,n}).
\end{equation}
Let
\begin{equation}
    Z_{t,n}^{\max}=(X_{t-1},\ldots,X_{t-P_n},-\varepsilon_{t-1},\ldots,-\varepsilon_{t-Q_n})^\top .
\end{equation}
By Lemmas \ref{lem:pop-cov-bound} and \ref{lem:finite-section-concentration},
\begin{equation}
    \max_{p,q}\lambda_{\max}\left(N^{-1}\sum_{t\in\mathcal I_n}Z_t^0(p,q)Z_t^0(p,q)^\top\right)
    \le
    \lambda_{\max}\left(N^{-1}\sum_{t\in\mathcal I_n}Z_{t,n}^{\max}(Z_{t,n}^{\max})^\top\right)
    =O_p(1).
\end{equation}
Consequently,
\begin{align}
    \sup_{p,q}\sup_{\gamma\in\mathcal K_{p,q}(K)}
    \|X-Z^0(p,q)\gamma\|_N
    &\le \|X\|_N+
    K\max_{p,q}\lambda_{\max}^{1/2}
    \left(N^{-1}\sum_{t\in\mathcal I_n}Z_t^0(p,q)Z_t^0(p,q)^\top\right) \\
    &=O_p(1).
\end{align}
For \(\gamma\in\mathcal K_{p,q}(K)\), define
\begin{equation}
    A_t(p,q,\gamma)=X_t-Z_t^0(p,q)^\top\gamma,
    \qquad
    B_t(p,q,\gamma)=\Delta X_t-\Delta Z_t(p,q)^\top\gamma.
\end{equation}
Then
\begin{equation}
    \widetilde X_t-\widetilde Z_t(p,q)^\top\gamma
    =A_t(p,q,\gamma)+B_t(p,q,\gamma),
\end{equation}
and \eqref{eq:s5-basic-delta-rates}--\eqref{eq:s5-deltaZ-norm-rate} imply
\begin{equation}
    \sup_{p,q}\sup_{\gamma\in\mathcal K_{p,q}(K)}
    \|B(p,q,\gamma)\|_N
    \le \|\Delta X\|_N+K\sup_{p,q}
    \left(N^{-1}\sum_{t\in\mathcal I_n}\|\Delta Z_t(p,q)\|_2^2\right)^{1/2}
    =O_p(M_n^{1/2}r_{e,n}).
\end{equation}
Therefore,
\begin{align}\label{eq:s5-loss-perturb-bound}
    &\max_{p,q}\sup_{\gamma\in\mathcal K_{p,q}(K)}
    \left|
    N^{-1}\sum_{t\in\mathcal I_n}(\widetilde X_t-\widetilde Z_t(p,q)^\top\gamma)^2
    -N^{-1}\sum_{t\in\mathcal I_n}(X_t-Z_t^0(p,q)^\top\gamma)^2
    \right|  \\
    &\quad\le
    2\sup_{p,q,\gamma}\|A(p,q,\gamma)\|_N
    \sup_{p,q,\gamma}\|B(p,q,\gamma)\|_N
    +\sup_{p,q,\gamma}\|B(p,q,\gamma)\|_N^2  \\
    &\quad=O_p(M_n^{1/2}r_{e,n}+M_nr_{e,n}^2).
\end{align}
Since \(|\inf f-\inf g|\le\sup|f-g|\), \eqref{eq:s5-loss-perturb-bound} gives
\begin{equation}\label{eq:s5-feasible-ideal-rss-bound}
    \max_{p,q}
    |\widehat\sigma_n^2(p,q)-\widehat\sigma_{0,n}^2(p,q)|
    =O_p(M_n^{1/2}r_{e,n}+M_nr_{e,n}^2).
\end{equation}
Combining \eqref{eq:s5-ideal-pop-bound}, \eqref{eq:s5-feasible-ideal-rss-bound}, and \eqref{eq:omega-n} proves \eqref{eq:uniform-resvar}.

For every \((p,q)\), \eqref{eq:s5-loss-decomposition} gives
\begin{equation}
    \sigma_{K,p,q}^{2,*}\ge \sigma_\varepsilon^2.
\end{equation}
Since the penalty window gives \(\omega_n\to0\), \eqref{eq:uniform-resvar} implies
\begin{equation}
    \Pp\left(
    \min_{0\le p\le P_n,\,0\le q\le Q_n}
    \widehat\sigma_n^2(p,q)\ge \sigma_\varepsilon^2/2
    \right)\to1.
\end{equation}
On this event,
\begin{align}
    \max_{p,q}
    \left|\log\widehat\sigma_n^2(p,q)-\log\sigma_{K,p,q}^{2,*}\right|
    &\le
    2\sigma_\varepsilon^{-2}
    \max_{p,q}
    |\widehat\sigma_n^2(p,q)-\sigma_{K,p,q}^{2,*}| \\
    &=O_p(\omega_n),
\end{align}
which proves \eqref{eq:uniform-log-resvar}.
\end{proof}

\subsection{Final proof of Theorem~\ref{thm:order-consistency}}

The preceding uniform approximation reduces the order-selection problem to the usual information-criterion comparison between underfitted and overfitted ARMA orders.  Let
\begin{equation}
    H_0=\mathrm{HIC}_n(p_0,q_0)
\end{equation}
and define the uniform logarithmic approximation error
\begin{equation}
    A_n=
    \max_{0\le p\le P_n,\,0\le q\le Q_n}
    \left|
    \log\widehat\sigma_n^2(p,q)-\log\sigma_{K,p,q}^{2,*}
    \right| .
\end{equation}
By Lemma \ref{lem:uniform-resvar},
\begin{equation}
    A_n=O_p(\omega_n).
\end{equation}
The penalty window \eqref{eq:sieve-penalty-window} implies
\begin{equation}\label{eq:s5-An-ratios}
    \frac{A_n}{\pi_n}=o_p(1).
\end{equation}
If \(\mathcal U_n\ne\varnothing\), then, since \(M_n\ge1\),
\begin{equation}\label{eq:s5-An-Delta-ratio}
    \frac{A_n}{\Delta_n}
    =O_p\left(\frac{\omega_n}{\Delta_n}\right)
    =O_p\left(\frac{\omega_n}{\pi_n}\frac{\pi_n}{\Delta_n}\right)
    =o_p(1),
    \qquad
    \frac{M_n\pi_n}{\Delta_n}\to0 .
\end{equation}

For every \((p,q)\in\mathcal U_n\),
\begin{align}
    \mathrm{HIC}_n(p,q)-H_0
    &=
    \log\widehat\sigma_n^2(p,q)-\log\widehat\sigma_n^2(p_0,q_0)
    +(p+q-p_0-q_0)\pi_n  \\
    &\ge
    \log\sigma_{K,p,q}^{2,*}-\log\sigma_\varepsilon^2
    -2A_n-2M_n\pi_n \\
    &\ge
    \Delta_n-2A_n-2M_n\pi_n .
\end{align}
Therefore \eqref{eq:s5-An-Delta-ratio} and the empty-set convention give
\begin{equation}\label{eq:s5-underfit-positive}
    \Pp\left\{
    \min_{(p,q)\in\mathcal U_n}
    \{\mathrm{HIC}_n(p,q)-H_0\}>0
    \right\}\to1.
\end{equation}

Set
\begin{equation}
    \mathcal O_n=
    \{(p,q):0\le p\le P_n,\ 0\le q\le Q_n,
    \ p\ge p_0,\ q\ge q_0,\ (p,q)\ne(p_0,q_0)\}.
\end{equation}
For every \((p,q)\in\mathcal O_n\), \eqref{eq:population-overfit} gives
\begin{equation}
    \sigma_{K,p,q}^{2,*}=\sigma_\varepsilon^2,
    \qquad
    p+q-p_0-q_0\ge1 .
\end{equation}
Hence
\begin{align}
    \mathrm{HIC}_n(p,q)-H_0
    &=
    \log\widehat\sigma_n^2(p,q)-\log\widehat\sigma_n^2(p_0,q_0)
    +(p+q-p_0-q_0)\pi_n \\
    &\ge
    -2A_n+\pi_n .
\end{align}
Together with \eqref{eq:s5-An-ratios} and the empty-set convention, this yields
\begin{equation}\label{eq:s5-overfit-positive}
    \Pp\left\{
    \min_{(p,q)\in\mathcal O_n}
    \{\mathrm{HIC}_n(p,q)-H_0\}>0
    \right\}\to1.
\end{equation}

For all large \(n\), the true order belongs to the candidate rectangle.  Every other candidate belongs to \(\mathcal U_n\cup\mathcal O_n\).  Combining \eqref{eq:s5-underfit-positive} and \eqref{eq:s5-overfit-positive},
\begin{equation}
    \Pp\left\{
    \min_{\substack{0\le p\le P_n,\,0\le q\le Q_n\\(p,q)\ne(p_0,q_0)}}
    \bigl[\mathrm{HIC}_n(p,q)-\mathrm{HIC}_n(p_0,q_0)\bigr]
    >0
    \right\}\to1.
\end{equation}
Thus every minimizer in \eqref{eq:selector} equals \((p_0,q_0)\) with probability tending to one, which proves \eqref{eq:main-consistency}.  \(\square\)

\end{document}